\documentclass[11pt]{article}

\usepackage[T1]{fontenc}
\usepackage[utf8]{inputenc}
\usepackage{lmodern}
\usepackage{microtype}
\usepackage{indentfirst}
\usepackage[a4paper,margin=1in]{geometry}
\usepackage{amsmath,amssymb,amsthm}
\usepackage{xcolor}
\usepackage{hyperref}
\usepackage[nameinlink,capitalise,noabbrev]{cleveref}

\hypersetup{
  colorlinks=true,
  linkcolor=blue,
  citecolor=blue,
  urlcolor=blue,
  pdftitle={A Gamma envelope and sharp moment inequalities for Gaussian quadratic forms},
  pdfauthor={Zhekai Pang}
}

\makeatletter
\renewcommand{\tagform@}[1]{%
  \maketag@@@{\color{blue}(\ignorespaces#1\unskip\@@italiccorr)}%
}
\makeatother

\newtheorem{theorem}{Theorem}[section]
\newtheorem{lemma}[theorem]{Lemma}
\newtheorem{corollary}[theorem]{Corollary}

\newcommand{\E}{\mathbb{E}}
\newcommand{\N}{\mathbb{N}}
\newcommand{\R}{\mathbb{R}}
\newcommand{\dd}{\,\mathrm{d}}

\title{A Gamma envelope and sharp moment inequalities
for Gaussian quadratic forms}
\author{Zhekai Pang\\
Universitat Pompeu Fabra\\
\texttt{zhekai.pang@upf.edu}}
\date{}

\begin{document}
\maketitle

\begin{abstract}
We study extremal absolute central moments of Gaussian quadratic forms under
a fixed Frobenius norm.  For a nonzero
real symmetric matrix $M$ and $G\sim N(0,I_n)$, we construct an explicit
centered difference of Gamma variables with the same mean, variance, and
third centered moment as $G^{\mathsf T}MG-\operatorname{tr}M$.  After
normalization, replacing the quadratic form by this Gamma difference does not
decrease $\E f$ for every $C^2$ test function $f$ such that $f''$ is convex
and $f$, $f'$, and $f''$ have polynomial growth.  In particular, it gives an
explicit upper bound for every absolute moment of order $p\ge3$.  We then
prove that, for every $p\ge4$, this bound
is maximized by the centered square $g^2-1$ of a single standard Gaussian
variable $g$.  The resulting sharp inequality is
\[
 \bigl\|G^{\mathsf T}MG-\operatorname{tr}M\bigr\|_p
 \le \|g^2-1\|_p\,\|M\|_{\mathrm F},
\]
with equality if and only if $M$ has rank one.
\end{abstract}



\section{Introduction}

Gaussian quadratic forms constitute a classical family of nonlinear Gaussian
functionals, and the spectral theorem represents every centered member of this
family as a finite weighted sum of independent centered chi-square variables.
Consequently, the eigenvalues of the defining matrix determine the
distribution, although this complete spectral description does not make the
dependence of absolute moments on the spectrum explicit.  In
Question~6 of~\cite{BNZ}, Bartczak, Nayar, and Zwara asked for the maximal
absolute central moment of a Gaussian quadratic form under a fixed variance.

The classical theory is largely concerned with exact distributions and
moments: Cochran characterized the quadratic forms having a chi-square
distribution~\cite{Cochran}, Imhof gave a method for computing the
distribution of a general form~\cite{Imhof}, and further distributional
results are collected in~\cite{MathaiProvost}.  Magnus derived formulas for
expectations of products of Gaussian quadratic forms~\cite{Magnus}, whereas a
separate line of work established estimates that are uniform over the
dimension and the defining matrix, including Whittle's moment
bounds~\cite{Whittle}, the Hanson--Wright tail
inequality~\cite{HansonWright,RudelsonVershynin}, and Lata\l a's two-sided
estimates for Gaussian chaoses~\cite{Latala}.

These general estimates do not determine the maximizing spectrum under a fixed
variance.  Chasapis, Singh, and Tkocz proved that, for every positive integer
$k$, the $k$th signed centered moment of a weighted sum of i.i.d. Gamma
variables with nonnegative coefficients is Schur-convex in the squared
coefficients~\cite{ChasapisSinghTkocz}.  For chi-square variables, this yields
rank-one maximization for nonnegative coefficients and even orders.  Our main
theorem allows arbitrary real coefficients and every
real $p\ge4$, and it also determines all equality cases.
Question~6 in~\cite{BNZ} is posed for every $p>2$.  The present paper
establishes rank-one extremality for every real $p\ge4$ and does not address
the range $2<p<4$.

Centered Gaussian quadratic forms also belong to the second Wiener chaos,
where Gamma and variance--Gamma variables arise as natural approximating
objects: Nourdin and Peccati obtained a moment criterion for convergence to a
centered Gamma variable~\cite{NourdinPeccati}, Gaunt developed Stein's method
for variance--Gamma approximation~\cite{Gaunt}, Eichelsbacher and Th\"ale
extended this method to Wiener chaos~\cite{EichelsbacherThale}, and Azmoodeh,
Eichelsbacher, and Th\"ale obtained optimal bounds on the second Wiener chaos
\cite{AzmoodehEichelsbacherThale}.  In this literature the Gamma-type variable
is an approximation target, whereas in the present paper it provides an explicit
upper envelope.

For each normalized quadratic form, \cref{thm:q6-gamma-reduction} constructs a
centered difference of two Gamma variables with matching first three centered
moments and proves that this replacement does not decrease the expectation for
every $C^2$ function whose second derivative is convex and whose value and first
two derivatives have polynomial growth.  This assumption is the
generalized fourth-order convexity condition used in the theory of higher
convex orders~\cite{DenuitLefevreShaked}, and it does not require a classical
fourth derivative.  Related comparison results for L\'evy processes based on
their jump measures are developed in~\cite{BergenthumRuschendorf}.  The proof
below is tailored to the generalized fourth-order convexity condition and
relies on the chord estimate
\eqref{eq:q6-convexity-sum}, which verifies that ordering directly and reduces
the full coefficient vector to a one-parameter Gamma difference matching the
first three centered moments.

Applying the comparison to $|x|^p$ gives Corollary~\ref{cor:q6-matrix-envelope} for
every $p\ge3$, with an upper bound that depends on a single scalar parameter.
The scalar comparison in \cref{thm:q6-scalar-all-high} shows that this bound
is maximized at the rank-one endpoint for every $p\ge4$.

Throughout the paper, $\|X\|_p=(\E|X|^p)^{1/p}$ for a random variable $X$,
and $\|M\|_{\mathrm F}=\sqrt{\operatorname{tr}(M^2)}$ for a real symmetric
matrix $M$.

\begin{theorem}\label{thm:q6-main}
For every $p\ge4$, for every $n\in\N^+$, and for every nonzero real symmetric
matrix $M\in\R^{n\times n}$, let $G\sim N(0,I_n)$ and
$g\sim N(0,1)$.  Then
\[
 \bigl\|G^{\mathsf T}MG-\operatorname{tr}M\bigr\|_p
 \le \|g^2-1\|_p\,\|M\|_{\mathrm F}.
\]
Equality holds if and only if $M$ has rank one, and the sharp constant is
\begin{equation}\label{eq:q6-main-constant}
 \|g^2-1\|_p
 =\left[\frac1{\sqrt{2\pi}}\int_0^\infty
 |x-1|^p x^{-1/2}\exp(-x/2)\dd x\right]^{1/p}.
\end{equation}
\end{theorem}

The remainder of the paper is organized as follows.
Section~\ref{sec:q6-gamma-envelope} diagonalizes the quadratic form, derives a common
representation of the moment generating functions, and proves the Gamma
envelope by a chord comparison and a compound-Poisson interpolation, while
Section~\ref{sec:q6-scalar-comparison} studies the resulting one-parameter family,
proves the comparison for $4<p\le6$ by direct integral calculations, extends
it to every larger exponent by two recurrences, and concludes with the proof
of \cref{thm:q6-main}.

\section{Gamma envelope}\label{sec:q6-gamma-envelope}

Let $G$ be a standard Gaussian vector in $\R^n$ and let $M$ be a nonzero
real symmetric matrix.  Let $d_1,\ldots,d_n$ be the eigenvalues of $M$.
By the spectral theorem and the rotational invariance of $G$, if
$g_1,\ldots,g_n$ are independent standard Gaussian variables, then
\[
 G^{\mathsf T}MG-\operatorname{tr}M
 \stackrel d=\sum\limits_{i=1}^n d_i(g_i^2-1).
\]
Let $a_i=d_i/\|M\|_{\mathrm F}$ and
$Q_a=\sum\limits_{i=1}^n a_i(g_i^2-1)$.  Then
$\sum\limits_{i=1}^n a_i^2=1$ and
\begin{equation}\label{eq:q6-diagonal-reduction}
 \frac{G^{\mathsf T}MG-\operatorname{tr}M}{\|M\|_{\mathrm F}}
 \stackrel d=Q_a.
\end{equation}

Let $\delta=\sum\limits_{i=1}^n a_i^3$.  Since
$|\sum\limits_{i=1}^n a_i^3|\le\sum\limits_{i=1}^n|a_i|^3
\le\sum\limits_{i=1}^n a_i^2=1$, we have $-1\le\delta\le1$.
We use the shape--rate parameterization for Gamma variables.  Let $A$ and
$B$ be independent, with
$A\sim\operatorname{Gamma}(\frac14(1+\delta),\frac12)$ and
$B\sim\operatorname{Gamma}(\frac14(1-\delta),\frac12)$, and let
$Z_\delta=A-B-\delta$.  A Gamma variable with shape zero is understood to
be zero.  In particular, $Z_1\stackrel d=g^2-1$ and
$Z_{-1}\stackrel d=1-g^2$.

The next theorem gives the reason for this choice.  The variable $Z_\delta$
keeps the first three centered moments of $Q_a$, and replacing $Q_a$ by
$Z_\delta$ does not decrease the expectation of any test function considered
in the theorem.

\begin{theorem}\label{thm:q6-gamma-reduction}
Let $f\in C^2(\R)$ and suppose that $f''$ is convex.  Suppose also that
there are constants $C>0$ and $m>0$ such that, for all $x\in\R$,
\[
 |f(x)|+|f'(x)|+|f''(x)|\le C(1+|x|^m).
\]
Then $\E f(Q_a)\le\E f(Z_\delta)$.  Moreover,
\[
 \E Q_a=\E Z_\delta=0,\qquad
 \E Q_a^2=\E Z_\delta^2=2,\qquad
 \E Q_a^3=\E Z_\delta^3=8\delta.
\]
In particular, for every $p\ge3$,
$\E|Q_a|^p\le\E|Z_\delta|^p$.
\end{theorem}

\begin{proof}
We first write the two MGFs in a common form.  This calculation identifies
the two endpoints of the interpolation used below.  Let
$R\sim\operatorname{Gamma}(2,\frac12)$, so $R$ has density
$x\exp(-x/2)/4$ on $(0,\infty)$.  For $|t|<1/2$, differentiation under the
expectation gives
\begin{align*}
 \frac{\dd}{\dd t}
 2\E\left[\frac{\exp(tR)-1-tR}{R^2}\right]
 &=2\E\left[\frac{\exp(tR)-1}{R}\right]\\
 &=\frac12\int_0^\infty(\exp(tx)-1)\exp(-x/2)\dd x\\
 &=\frac1{1-2t}-1\\
 &=\frac{\dd}{\dd t}\left[-t-\frac12\log(1-2t)\right].
\end{align*}
Both expressions inside the derivatives vanish at $t=0$.  Integrating
from $0$ to $t$ gives
\begin{equation}\label{eq:q6-basic-mgf-integral}
 -t-\frac12\log(1-2t)
 =2\E\left[\frac{\exp(tR)-1-tR}{R^2}\right].
\end{equation}

For $|t|<1/2$, we have
\begin{align}
 \E[\exp(tQ_a)]
 &=\E\left[\prod\limits_{i=1}^n
   \exp\bigl(ta_i(g_i^2-1)\bigr)\right]\notag\\
 &=\prod\limits_{i=1}^n
   \int_{-\infty}^{\infty}
   \exp\bigl(ta_i(x^2-1)\bigr)
   \frac1{\sqrt{2\pi}}\exp(-x^2/2)\dd x\notag\\
 &=\prod\limits_{i=1}^n \exp(-ta_i)
   \int_{-\infty}^{\infty}
   \frac1{\sqrt{2\pi}}\exp\left(-\frac{1}{2}(1-2ta_i)x^2\right)\dd x\notag\\
 &=\prod\limits_{i=1}^n \exp(-ta_i)(1-2ta_i)^{-1/2}\notag\\
 &=\exp\left\{\sum\limits_{i=1}^n
   \left[-ta_i-\frac12\log(1-2ta_i)\right]\right\}\notag\\
 &=\exp\left\{2\sum\limits_{i=1}^n
   \E\left[\frac{\exp(ta_iR)-1-ta_iR}{R^2}\right]\right\},
   \label{eq:q6-Qa-mgf}
\end{align}
\begin{align}
 \E[\exp(tZ_\delta)]
 &=\exp(-\delta t)\E[\exp(tA)]\E[\exp(-tB)]\notag\\
 &=\exp(-\delta t)(1-2t)^{-\frac14(1+\delta)}
   (1+2t)^{-\frac14(1-\delta)}\notag\\
 &=\exp\left\{\frac{1}{2}(1+\delta)
 \left[-t-\frac12\log(1-2t)\right]
 +\frac{1}{2}(1-\delta)
 \left[t-\frac12\log(1+2t)\right]\right\}\notag\\
 &=\exp\left\{(1+\delta)\E\left[
   \frac{\exp(tR)-1-tR}{R^2}\right]
 +(1-\delta)\E\left[
   \frac{\exp(-tR)-1+tR}{R^2}\right]\right\}.
   \label{eq:q6-Zdelta-mgf}
\end{align}
The final equality in \eqref{eq:q6-Qa-mgf} follows from
\eqref{eq:q6-basic-mgf-integral} with $t$ replaced by $ta_i$, followed by
summation from $i=1$ to $n$.  The final equality in
\eqref{eq:q6-Zdelta-mgf} follows from \eqref{eq:q6-basic-mgf-integral}
applied at $t$ and at $-t$.

Equations~\eqref{eq:q6-Qa-mgf} and \eqref{eq:q6-Zdelta-mgf} have the same
form.  Their exponents are integrals of
$\exp(tx)-1-tx$, with jump sizes $a_iR$ in the first formula and the two
endpoint jump sizes $R$ and $-R$ in the second.  The measures represented
by these formulas have infinite mass near zero.  We therefore remove the
values $R\le\varepsilon$, which leaves finite measures.  We shall construct
a centered compound-Poisson variable from each truncated measure,
interpolate between the two variables, and prove that the expectation of
$f$ cannot decrease along the interpolation.  At the end of the proof we
let $\varepsilon\downarrow0$ to recover $Q_a$ and $Z_\delta$.

For $\varepsilon>0$, define two finite measures as follows.  For every
bounded measurable function $h$,
\begin{align*}
 \int_{-\infty}^{\infty}h(x)\nu_{0,\varepsilon}(\dd x)
 &=2\E\left[\frac{\mathbf1_{\{R>\varepsilon\}}}{R^2}
 \sum\limits_{\substack{1\le i\le n\\a_i\ne0}}h(a_iR)\right],\\
 \int_{-\infty}^{\infty}h(x)\nu_{1,\varepsilon}(\dd x)
 &=2\E\left[\frac{\mathbf1_{\{R>\varepsilon\}}}{R^2}
 \left(\frac{1}{2}(1+\delta)h(R)
 +\frac{1}{2}(1-\delta)h(-R)\right)\right].
\end{align*}
Taking $h(x)=1$ and using $R^{-2}\le\varepsilon^{-2}$ on
$\{R>\varepsilon\}$ gives
\begin{align*}
 \nu_{0,\varepsilon}(\R)
 &=2\E\left[\frac{\mathbf1_{\{R>\varepsilon\}}}{R^2}
 \sum\limits_{\substack{1\le i\le n\\a_i\ne0}}1\right]
 \le \frac{2n}{\varepsilon^2},\\
 \nu_{1,\varepsilon}(\R)
 &=2\E\left[\frac{\mathbf1_{\{R>\varepsilon\}}}{R^2}
 \left(\frac{1}{2}(1+\delta)+\frac{1}{2}(1-\delta)\right)\right]
 =2\E\left[\frac{\mathbf1_{\{R>\varepsilon\}}}{R^2}\right]
 \le \frac{2}{\varepsilon^2}.
\end{align*}
Using bounded monotone approximations to $h(x)=|x|$ and
$R^{-1}\le\varepsilon^{-1}$, the monotone convergence theorem gives
\[
 \int_{-\infty}^{\infty}|x|\nu_{0,\varepsilon}(\dd x)
 \le \frac{2}{\varepsilon}\sum\limits_{i=1}^n|a_i|,
 \qquad
 \int_{-\infty}^{\infty}|x|\nu_{1,\varepsilon}(\dd x)
 \le \frac{2}{\varepsilon}.
\]
Thus both endpoint measures are finite and have finite first moments.
Because at least one $a_i$ is nonzero and
$\mathbb{P}(R>\varepsilon)>0$, both total masses are positive.
The first measure describes the truncated jumps $a_iR$ in
\eqref{eq:q6-Qa-mgf}, and the second describes the truncated jumps $R$ and
$-R$ in \eqref{eq:q6-Zdelta-mgf}.  For $0\le r\le1$, let
$\nu_{r,\varepsilon}=(1-r)\nu_{0,\varepsilon}
+r\nu_{1,\varepsilon}$.

We now give the construction of the corresponding centered
compound-Poisson variable.  Fix $r$ and $\varepsilon$.  Let $K$ be a
Poisson variable with mean $\nu_{r,\varepsilon}(\R)$, and let
$J_1,J_2,\ldots$ be independent variables with common distribution
$\frac{\nu_{r,\varepsilon}}{\nu_{r,\varepsilon}(\R)}$, independent of
$K$.  Let
\[
 Y_{r,\varepsilon}
 =\sum\limits_{j=1}^KJ_j
 -\int_{-\infty}^{\infty}x\nu_{r,\varepsilon}(\dd x).
\]
The centering follows from
\begin{align*}
 \E\left[\sum\limits_{j=1}^KJ_j\right]
 &=\E[K]\,\E[J_1]
 =\nu_{r,\varepsilon}(\R)
   \frac{\displaystyle\int_{-\infty}^{\infty}
   x\nu_{r,\varepsilon}(\dd x)}{\nu_{r,\varepsilon}(\R)}
 =\int_{-\infty}^{\infty}x\nu_{r,\varepsilon}(\dd x).
\end{align*}
Thus $\E Y_{r,\varepsilon}=0$.  For $|t|<1/2$,
using $\E[s^K]=\exp\{\nu_{r,\varepsilon}(\R)(s-1)\}$ gives
\begin{align}
 \E[\exp(tY_{r,\varepsilon})]
 &=\exp\left\{-t\int_{-\infty}^{\infty}
      x\nu_{r,\varepsilon}(\dd x)\right\}
   \E\left[\exp\left(t\sum\limits_{j=1}^KJ_j\right)\right]\notag\\
 &=\exp\left\{-t\int_{-\infty}^{\infty}
      x\nu_{r,\varepsilon}(\dd x)\right\}
   \E\left[\E\left[\left.
     \exp\left(t\sum\limits_{j=1}^KJ_j\right)\right|K\right]\right]\notag\\
 &=\exp\left\{-t\int_{-\infty}^{\infty}
      x\nu_{r,\varepsilon}(\dd x)\right\}
   \E\left[\bigl(\E[\exp(tJ_1)]\bigr)^K\right]\notag\\
 &=\exp\left\{-t\int_{-\infty}^{\infty}
      x\nu_{r,\varepsilon}(\dd x)
      +\int_{-\infty}^{\infty}(\exp(tx)-1)
       \nu_{r,\varepsilon}(\dd x)\right\}\notag\\
 &=\exp\left\{\int_{-\infty}^{\infty}
      (\exp(tx)-1-tx)\nu_{r,\varepsilon}(\dd x)\right\}.
      \label{eq:q6-truncated-mgf}
\end{align}

The next estimate gives a uniform exponential-moment bound for
$Y_{r,\varepsilon}$.  We use \eqref{eq:q6-truncated-mgf} at $t=1/4$ and
$t=-1/4$.
For either choice of sign,
$\exp(\pm x/4)-1\mp x/4\ge0$ for all $x\in\R$.  Removing the factor
$\mathbf1_{\{R>\varepsilon\}}$ from the definition of
$\nu_{r,\varepsilon}$ can therefore only increase the exponent in
\eqref{eq:q6-truncated-mgf}.  Hence,
\begin{align*}
 \E[\exp(\pm Y_{r,\varepsilon}/4)]
 &=\exp\left\{\int_{-\infty}^{\infty}
   (\exp(\pm x/4)-1\mp x/4)
   \nu_{r,\varepsilon}(\dd x)\right\}\\
 &\le \exp\left\{(1-r)\log\E[\exp(\pm Q_a/4)]
       +r\log\E[\exp(\pm Z_\delta/4)]\right\}\\
 &\le\max\{\E[\exp(\pm Q_a/4)],
                 \E[\exp(\pm Z_\delta/4)]\}.
\end{align*}
Since $\exp(|y|/4)\le\exp(y/4)+\exp(-y/4)$, it follows that
\begin{equation}\label{eq:q6-uniform-exponential-moment}
 \sup_{0\le r\le1,\ \varepsilon>0}
 \E[\exp(|Y_{r,\varepsilon}|/4)]<\infty.
\end{equation}

We now compare the expectations at the two endpoints.  Choose
$\chi\in C_c^\infty(\R)$ such that
$0\le\chi\le1$, $\chi=1$ on $[-1,1]$, and $\chi=0$ outside $[-2,2]$.
For $N\ge1$, let $f_N(x)=f(x)\chi(x/N)$.  Then
\[
 f_N''(y)=f''(y)\chi(y/N)
 +\frac2Nf'(y)\chi'(y/N)+\frac1{N^2}f(y)\chi''(y/N).
\]
For every fixed $y\in\R$ and every $N>|y|$, $f_N''(y)=f''(y)$.
The polynomial bounds on $f,f'$, and $f''$ therefore give, for a constant
$C_m$ independent of $N$,
$|f_N(y)|+|f_N''(y)|\le C_m(1+|y|^m)$.

Fix $N$.  For fixed $r$ and every $h$ such that $r+h\in[0,1]$,
$\nu_{r+h,\varepsilon}=\nu_{r,\varepsilon}
+h(\nu_{1,\varepsilon}-\nu_{0,\varepsilon})$, and hence
\[
 \int_{-\infty}^{\infty}x\nu_{r+h,\varepsilon}(\dd x)
 =\int_{-\infty}^{\infty}x\nu_{r,\varepsilon}(\dd x)
   +h\int_{-\infty}^{\infty}x
   (\nu_{1,\varepsilon}-\nu_{0,\varepsilon})(\dd x).
\]
Thus a point configuration $\mu$ with intensity
$\nu_{r+h,\varepsilon}$ is evaluated by
\[
 \mu\longmapsto
 f_N\left(\int_{-\infty}^{\infty}x\mu(\dd x)
 -\int_{-\infty}^{\infty}x\nu_{r,\varepsilon}(\dd x)
 -h\int_{-\infty}^{\infty}x
 (\nu_{1,\varepsilon}-\nu_{0,\varepsilon})(\dd x)\right).
\]
The Poisson perturbation formula
\cite[Theorem~19.1]{LastPenrose} differentiates the change of intensity,
and the ordinary chain rule differentiates the $h$-dependent centering term.
The perturbation $\nu_{1,\varepsilon}-\nu_{0,\varepsilon}$ is a finite signed
measure, as allowed in the cited theorem.  Since
$\nu_{r+h,\varepsilon}=(1-r-h)\nu_{0,\varepsilon}
+(r+h)\nu_{1,\varepsilon}$, the perturbed intensity remains positive whenever
$r+h\in[0,1]$.
Because $f_N$ and $f_N'$ are bounded, the two derivatives may be combined
under the expectation.
At $h=0$ they give
\begin{align}
 \frac{\dd}{\dd r}\E f_N(Y_{r,\varepsilon})
 &=\int_{-\infty}^{\infty}
   \E[f_N(Y_{r,\varepsilon}+x)-f_N(Y_{r,\varepsilon})]
   \frac{\partial\nu_{r,\varepsilon}}{\partial r}(\dd x)
   -\E f_N'(Y_{r,\varepsilon})
   \int_{-\infty}^{\infty}x
   \frac{\partial\nu_{r,\varepsilon}}{\partial r}(\dd x)\notag\\
 &=\int_{-\infty}^{\infty}
   \E[f_N(Y_{r,\varepsilon}+x)-f_N(Y_{r,\varepsilon})
   -xf_N'(Y_{r,\varepsilon})]
   \frac{\partial\nu_{r,\varepsilon}}{\partial r}(\dd x).
   \label{eq:q6-poisson-derivative}
\end{align}
At $r=0$ and $r=1$, we use the corresponding one-sided derivatives, as
allowed in \cite[Theorem~19.1]{LastPenrose}.  Taylor's formula gives
\begin{equation}\label{eq:q6-taylor-remainder}
 f_N(y+x)-f_N(y)-xf_N'(y)
 =x^2\int_0^1(1-s)f_N''(y+sx)\dd s.
\end{equation}
Since $\partial\nu_{r,\varepsilon}/\partial r
=\nu_{1,\varepsilon}-\nu_{0,\varepsilon}$, substituting
\eqref{eq:q6-taylor-remainder} into \eqref{eq:q6-poisson-derivative} and
then using the definitions of the two endpoint measures gives
\begin{align}
 \frac{\dd}{\dd r}\E f_N(Y_{r,\varepsilon})
 &=\int_{-\infty}^{\infty}
 \E\left[x^2\int_0^1(1-s)
 f_N''(Y_{r,\varepsilon}+sx)\dd s\right]
 (\nu_{1,\varepsilon}-\nu_{0,\varepsilon})(\dd x)\notag\\
 &=2\E\Biggl[\mathbf1_{\{R>\varepsilon\}}\int_0^1(1-s)
 \Biggl(\frac{1}{2}(1+\delta)f_N''(Y_{r,\varepsilon}+sR)\notag\\
 &\qquad{}+\frac{1}{2}(1-\delta)f_N''(Y_{r,\varepsilon}-sR)
 -\sum\limits_{i=1}^n a_i^2
 f_N''(Y_{r,\varepsilon}+sa_iR)\Biggr)\dd s\Biggr],
 \label{eq:q6-truncated-derivative}
\end{align}
where $R$ is independent of $Y_{r,\varepsilon}$.

We integrate \eqref{eq:q6-truncated-derivative} from $r=0$ to $r=1$ and
then let $N\to\infty$.
For every $j\in\{0,1\}$, $f_N(Y_{j,\varepsilon})\longrightarrow
f(Y_{j,\varepsilon})$ and
$|f_N(Y_{j,\varepsilon})|\le C_m(1+|Y_{j,\varepsilon}|^m)$.
For every $|c|\le1$ and $0\le s\le1$,
$f_N''(Y_{r,\varepsilon}+scR)\longrightarrow
f''(Y_{r,\varepsilon}+scR)$.
Since $-1\le\delta\le1$ and $\sum\limits_{i=1}^n a_i^2=1$, the triangle
inequality and the polynomial bound on $f_N''$ give
\begin{align*}
 {}&2\mathbf1_{\{R>\varepsilon\}}(1-s)
 \left|\frac{1}{2}(1+\delta)f_N''(Y_{r,\varepsilon}+sR)
 +\frac{1}{2}(1-\delta)f_N''(Y_{r,\varepsilon}-sR)
 -\sum\limits_{i=1}^n a_i^2f_N''(Y_{r,\varepsilon}+sa_iR)\right|\\
 \le{}&2\mathbf1_{\{R>\varepsilon\}}(1-s)
 C_m(1+|Y_{r,\varepsilon}|^m+R^m)
 \left(\frac{1}{2}(1+\delta)+\frac{1}{2}(1-\delta)
 +\sum\limits_{i=1}^n a_i^2\right)\\
 ={}&4\mathbf1_{\{R>\varepsilon\}}(1-s)
 C_m(1+|Y_{r,\varepsilon}|^m+R^m)\\
 \le{}&4C_m(1+|Y_{r,\varepsilon}|^m+R^m).
\end{align*}
Equation~\eqref{eq:q6-uniform-exponential-moment} and the Gamma moments give
\[
 \sup_{0\le r\le1,\ \varepsilon>0}
 \E[1+|Y_{r,\varepsilon}|^m+R^m]<\infty.
\]
The dominated convergence theorem therefore applies to the two endpoint
expectations and to the integral of \eqref{eq:q6-truncated-derivative}.  We obtain
\begin{align}
 \E f(Y_{1,\varepsilon})-\E f(Y_{0,\varepsilon})
 &=2\int_0^1\E\Biggl[\mathbf1_{\{R>\varepsilon\}}
   \int_0^1(1-s)
   \Biggl(\frac{1}{2}(1+\delta)f''(Y_{r,\varepsilon}+sR)\notag\\
 &\qquad{}
   +\frac{1}{2}(1-\delta)f''(Y_{r,\varepsilon}-sR)
   -\sum\limits_{i=1}^n a_i^2
    f''(Y_{r,\varepsilon}+sa_iR)\Biggr)\dd s\Biggr]\dd r.
    \label{eq:q6-integrated-derivative}
\end{align}

Since $|a_i|\le1$ for every $1\le i\le n$, the convexity of $f''$ gives,
for every $y\in\R$, $R>0$, and $0\le s\le1$,
\[
 f''(y+sa_iR)
 \le\frac{1}{2}(1+a_i)f''(y+sR)
    +\frac{1}{2}(1-a_i)f''(y-sR).
\]
Therefore,
\begin{align}
 \sum\limits_{i=1}^n a_i^2f''(y+sa_iR)
 &\le \frac{1}{2}\sum\limits_{i=1}^na_i^2(1+a_i)f''(y+sR)
     +\frac{1}{2}\sum\limits_{i=1}^na_i^2(1-a_i)f''(y-sR)\notag\\
 &=\frac{1}{2}\left(\sum\limits_{i=1}^na_i^2
       +\sum\limits_{i=1}^na_i^3\right)f''(y+sR)
   +\frac{1}{2}\left(\sum\limits_{i=1}^na_i^2
       -\sum\limits_{i=1}^na_i^3\right)f''(y-sR)\notag\\
 &=\frac{1}{2}(1+\delta)f''(y+sR)
   +\frac{1}{2}(1-\delta)f''(y-sR).
   \label{eq:q6-convexity-sum}
\end{align}
Substituting \eqref{eq:q6-convexity-sum} into
\eqref{eq:q6-integrated-derivative} gives
\begin{equation}\label{eq:q6-truncated-comparison}
 \E f(Y_{0,\varepsilon})\le\E f(Y_{1,\varepsilon}).
\end{equation}

It remains to let $\varepsilon\downarrow0$.  In passing from
\eqref{eq:q6-truncated-derivative} to
\eqref{eq:q6-truncated-comparison}, we integrated in $r$ and removed the
$N$-cutoff in the test function.  Thus only the two endpoint expectations
remain to be considered in this limit.  Since
$\exp(tx)-1-tx\ge0$, the monotone convergence theorem applied to
\eqref{eq:q6-truncated-mgf} shows that,
at $r=0$ and $r=1$, the MGFs converge for every $|t|<1/2$ to
\eqref{eq:q6-Qa-mgf} and \eqref{eq:q6-Zdelta-mgf},
respectively.  The continuity theorem for MGFs gives
$Y_{0,\varepsilon}\longrightarrow Q_a$ and
$Y_{1,\varepsilon}\longrightarrow Z_\delta$ in distribution.
For every fixed $m$ and all $y\in\R$,
$(1+|y|^m)^2\le C_m\exp(|y|/4)$.  The polynomial bound on $f$ and
\eqref{eq:q6-uniform-exponential-moment} therefore show that
$f(Y_{0,\varepsilon})$ and $f(Y_{1,\varepsilon})$ are uniformly integrable.
Consequently, as $\varepsilon\downarrow0$,
$\E f(Y_{0,\varepsilon})\longrightarrow\E f(Q_a)$ and
$\E f(Y_{1,\varepsilon})\longrightarrow\E f(Z_\delta)$.
Taking the limit $\varepsilon\downarrow0$ in
\eqref{eq:q6-truncated-comparison} gives
$\E f(Q_a)\le\E f(Z_\delta)$.

It remains to verify the moment assertions.  Since
$\E(g_i^2-1)^2=2$ and $\E(g_i^2-1)^3=8$, independence and centering give
\begin{align*}
 \E Q_a&=0,\qquad
 \E Q_a^2=2\sum\limits_{i=1}^na_i^2=2,\qquad
 \E Q_a^3=8\sum\limits_{i=1}^na_i^3=8\delta.
\end{align*}
A Gamma variable with shape $\alpha$ and rate $1/2$ has mean $2\alpha$,
variance $4\alpha$, and third centered moment $16\alpha$.  Since $A$ and
$B$ are independent and $Z_\delta=A-B-\delta$, their shapes give
$\E Z_\delta=0$, $\E Z_\delta^2=2$, and $\E Z_\delta^3=8\delta$.
This proves the three moment identities in the statement.

Finally, for $f(x)=|x|^p$ with $p\ge3$, the function belongs to $C^2(\R)$ and
one has
$f''(x)=p(p-1)|x|^{p-2}$, which is convex.  This proves the final assertion.
\end{proof}

The Gamma envelope also gives a concrete comparison beyond absolute moments.
For every $u\in\R$, the function $x\mapsto(\max\{x-u,0\})^3$ is $C^2$, has
polynomial growth, and has convex second derivative $6\max\{x-u,0\}$.  Hence
\cref{thm:q6-gamma-reduction} also gives
\[
 \E[(\max\{Q_a-u,0\})^3]\le\E[(\max\{Z_\delta-u,0\})^3].
\]
The strict scalar comparison in Section~\ref{sec:q6-scalar-comparison} will also
determine the equality cases in \cref{thm:q6-main}.

\begin{corollary}\label{cor:q6-matrix-envelope}
Let $M\in\R^{n\times n}$ be a nonzero real symmetric matrix, let
$G\sim N(0,I_n)$, and let
$\delta=\frac{\operatorname{tr}(M^3)}{\|M\|_{\mathrm F}^3}$.
Then $|\delta|\le1$.  For every function $f$ satisfying the assumptions
of \cref{thm:q6-gamma-reduction},
\begin{equation}\label{eq:q6-matrix-test-comparison}
 \E f\left(\frac{G^{\mathsf T}MG-\operatorname{tr}M}
                  {\|M\|_{\mathrm F}}\right)
 \le \E f(Z_\delta).
\end{equation}
In particular, for every $p\ge3$,
\begin{equation}\label{eq:q6-matrix-envelope}
 \E\left|G^{\mathsf T}MG-\operatorname{tr}M\right|^p
 \le \|M\|_{\mathrm F}^p\E|Z_\delta|^p.
\end{equation}
The variable $\|M\|_{\mathrm F}Z_\delta$ has the
same mean, variance, and third centered moment as
$G^{\mathsf T}MG-\operatorname{tr}M$.
\end{corollary}

\begin{proof}
Let $d_1,\ldots,d_n$ be the eigenvalues of $M$ and let
$a_i=d_i/\|M\|_{\mathrm F}$.  The identities
$\sum\limits_{i=1}^n d_i^2=\|M\|_{\mathrm F}^2$ and
$\sum\limits_{i=1}^n d_i^3=\operatorname{tr}(M^3)$ give
$\sum\limits_{i=1}^n a_i^2=1$ and
$\sum\limits_{i=1}^n a_i^3=\delta$.  Hence
$|\delta|=|\sum\limits_{i=1}^n a_i^3|
\le\sum\limits_{i=1}^n|a_i|^3
\le\sum\limits_{i=1}^n a_i^2=1$.
Thus the coefficient vector in \eqref{eq:q6-diagonal-reduction} is
normalized, and its parameter in \cref{thm:q6-gamma-reduction} is exactly
$\delta$.  Applying that theorem to \eqref{eq:q6-diagonal-reduction} gives
\[
 \E f\left(\frac{G^{\mathsf T}MG-\operatorname{tr}M}
                  {\|M\|_{\mathrm F}}\right)
 =\E f(Q_a)
 \le \E f(Z_\delta).
\]

For $p\ge3$, applying \eqref{eq:q6-matrix-test-comparison} to
$f(x)=|x|^p$ gives
\[
 \frac{\E\left|G^{\mathsf T}MG-\operatorname{tr}M\right|^p}
 {\|M\|_{\mathrm F}^p}
 =\E|Q_a|^p\le\E|Z_\delta|^p,
\]
which is \eqref{eq:q6-matrix-envelope}.

Finally, \cref{thm:q6-gamma-reduction} gives the same first three centered
moments for $Q_a$ and $Z_\delta$.  Using
\eqref{eq:q6-diagonal-reduction} and multiplying the moments of orders one,
two, and three by the corresponding powers of $\|M\|_{\mathrm F}$ shows
that $G^{\mathsf T}MG-\operatorname{tr}M$ and
$\|M\|_{\mathrm F}Z_\delta$ have the same first three centered moments.
Their common third centered moment is
$8\|M\|_{\mathrm F}^3\delta=8\operatorname{tr}(M^3)$.
\end{proof}

\section{Scalar comparison}\label{sec:q6-scalar-comparison}

It remains to compare $Z_\delta$ with the endpoint $g^2-1$.  Since
$Z_{-\delta}\stackrel d=-Z_\delta$, its absolute moments depend on $\delta$
only through $|\delta|$.  We first handle
$4\le p\le6$ by direct calculation and then derive recurrences for all
larger exponents.

\begin{lemma}\label{lem:q6-endpoints}
Let $g_1,\ldots,g_n$ be independent standard Gaussian variables, and let
$g$ be a standard Gaussian variable.  If
$\sum\limits_{i=1}^n a_i^2=1$, then
\[
 \E\left(\sum\limits_{i=1}^n a_i(g_i^2-1)\right)^4
 \le60=\E(g^2-1)^4.
\]
Equality holds if and only if, for some $1\le k\le n$, one has
$|a_k|=1$ and $a_i=0$ for every $i\ne k$.
\end{lemma}

\begin{proof}
The standard fourth-moment expansion
for independent centered $\chi_1^2$ variables gives
\begin{align*}
 \E\left(\sum\limits_{i=1}^n a_i(g_i^2-1)\right)^4
 &=60\sum\limits_{i=1}^n a_i^4
   +24\sum\limits_{1\le i<j\le n}a_i^2a_j^2\\
 &=12\left(\sum\limits_{i=1}^n a_i^2\right)^2
   +48\sum\limits_{i=1}^n a_i^4\\
 &=12+48\sum\limits_{i=1}^n a_i^4\\
 &\le60.
\end{align*}
Equality holds precisely when $a_i^2a_j^2=0$ for every
$1\le i<j\le n$.  Since the squares sum to one, this is equivalent to
$|a_k|=1$ for one index $k$ and $a_i=0$ for every $i\ne k$.
\end{proof}

\begin{theorem}
\label{thm:gamma-phase-comparison}
For every $0\le\delta\le1$, the following statements hold.
\begin{enumerate}
\item The fourth moment satisfies
\[
 \E|Z_\delta|^4=\E|g^2-1|^4=60.
\]
\item For every $4<p<6$,
\[
 \E|Z_\delta|^p\le\E|g^2-1|^p.
\]
Equality holds if and only if $\delta=1$.
\item The sixth moment satisfies
\[
 \E|Z_\delta|^6=5400+640\delta^2\le6040=\E|g^2-1|^6.
\]
\end{enumerate}
\end{theorem}

\begin{proof}
Fix $0\le\delta\le1$.  For $|t|<1/2$, independence and the Gamma MGF give
\[
 \E[\exp(tZ_\delta)]
 =\exp(-\delta t)(1-2t)^{-\frac14(1+\delta)}
 (1+2t)^{-\frac14(1-\delta)}.
\]
\medskip
\noindent\emph{The first six moments.}
For $|t|<1/2$, the power-series expansions of the logarithms give, as
$t\to0$,
\begin{align*}
 \log\E[\exp(tZ_\delta)]
 &=-\delta t-\frac{1}{4}(1+\delta)\log(1-2t)
   -\frac{1}{4}(1-\delta)\log(1+2t)\\
 &=-\delta t
   +\frac{1}{4}(1+\delta)\sum\limits_{j=1}^{6}\frac{(2t)^j}{j}
   +\frac{1}{4}(1-\delta)\sum\limits_{j=1}^{6}
      \frac{(-1)^j(2t)^j}{j}+O(t^7)\\
 &=t^2+\frac{4}{3}\delta t^3+2t^4
  +\frac{16}{5}\delta t^5+\frac{16}{3}t^6+O(t^7).
\end{align*}
Here we used
$\log(1-z)=-\sum\limits_{j=1}^{\infty}z^j/j$ and
$\log(1+z)=\sum\limits_{j=1}^{\infty}(-1)^{j+1}z^j/j$.  The expansion
$\exp(z)=1+z+\frac{1}{2}z^2+\frac{1}{6}z^3+O(z^4)$ gives
\begin{align*}
 \E[\exp(tZ_\delta)]
 &=1+t^2+\frac{4}{3}\delta t^3+\left(2+\frac{1}{2}\right)t^4
   +\left(\frac{16}{5}+\frac{4}{3}\right)\delta t^5
   +\left(\frac{16}{3}+2+\frac{1}{6}+\frac{8}{9}\delta^2\right)t^6+O(t^7)\\
 &=1+t^2+\frac{4}{3}\delta t^3+\frac52t^4
  +\frac{68}{15}\delta t^5
  +\left(\frac{15}{2}+\frac{8}{9}\delta^2\right)t^6+O(t^7).
\end{align*}
Comparing the coefficient of $t^j$ with
$\E[Z_\delta^j]/j!$ gives
\begin{equation}\label{eq:q6-low-moments}
 \E Z_\delta^2=2,\quad \E Z_\delta^3=8\delta,\quad
 \E Z_\delta^4=60,\quad \E Z_\delta^5=544\delta,\quad
 \E Z_\delta^6=5400+640\delta^2.
\end{equation}
The fourth-moment identity in \eqref{eq:q6-low-moments} proves part~(1).
The sixth-moment identity gives
$\E|Z_\delta|^6=5400+640\delta^2\le6040=\E|Z_1|^6$, which proves
part~(3).  It remains to prove part~(2) for $4<p<6$.

\medskip
\noindent\emph{A positive integral.}
To compare the moments of $Z_\delta$ and $Z_1$, we first determine the sign
of the difference between their cosine transforms.
The characteristic function of a Gamma variable with shape $r$ and rate
$1/2$ is $(1-2iu)^{-r}$.  For every $u\in\R$, let
$m(u)=(1+4u^2)^{-1/4}$ and $b(u)=u-\frac{1}{2}\arctan(2u)$.  Since
\[
 1-2iu=\sqrt{1+4u^2}\exp\bigl(-i\arctan(2u)\bigr),\qquad
 1+2iu=\sqrt{1+4u^2}\exp\bigl(i\arctan(2u)\bigr),
\]
independence gives
\begin{align*}
 \E[\exp(iuZ_\delta)]
 &=\exp(-iu\delta)(1-2iu)^{-\frac14(1+\delta)}
   (1+2iu)^{-\frac14(1-\delta)}\\
 &=m(u)\exp\left[-i\delta\left(u-\frac{1}{2}\arctan(2u)\right)\right]\\
 &=m(u)\exp\bigl(-i\delta b(u)\bigr).
\end{align*}
Taking real and imaginary parts gives
\begin{equation}\label{eq:q6-trig-transforms}
 \E\cos(uZ_\delta)=m(u)\cos(\delta b(u)),\qquad
 \E\sin(uZ_\delta)=-m(u)\sin(\delta b(u)).
\end{equation}
The moment comparison below uses the difference between the cosine transforms
corresponding to the parameters $\delta$ and $1$.  Using
\eqref{eq:q6-trig-transforms} with these two parameters gives, for every
$u\in\R$,
\[
 \E\cos(uZ_\delta)-\E\cos(uZ_1)
 =m(u)[\cos(\delta b(u))-\cos b(u)].
\]
For $0\le\delta<1$ and $0<p<6$, we first verify that the integral of the
right-hand side against $u^{-p-1}$ over $(0,\infty)$ is absolutely
convergent.  The expansion
$\arctan(2u)=2u-8u^3/3+O(u^5)$ gives
$b(u)=4u^3/3+O(u^5)$ at zero.  Since
$\cos(\delta b(u))-\cos b(u)=O(b(u)^2)$, for all $0\le\delta<1$,
\[
 m(u)[\cos(\delta b(u))-\cos b(u)]u^{-p-1}=O(u^{5-p})
 \quad(u\downarrow0).
\]
At infinity, boundedness of the cosine functions and
$m(u)=O(u^{-1/2})$ give an $O(u^{-p-3/2})$ bound.  The estimates at zero and
infinity show that
\begin{equation}\label{eq:q6-phase-absolute}
 \int_0^\infty m(u)|\cos(\delta b(u))-\cos b(u)|u^{-p-1}\dd u<\infty.
\end{equation}
We now prove that the corresponding integral without the absolute value is
positive:
\begin{equation}\label{eq:q6-phase-positive}
 \int_0^\infty m(u)[\cos(\delta b(u))-\cos b(u)]u^{-p-1}\dd u>0.
\end{equation}
For $u>0$, one has $b'(u)=4u^2/(1+4u^2)>0$, while $b(0)=0$ and
$b(u)\to\infty$ as $u\to\infty$.  Since $b$ is continuous and strictly
increasing, the intermediate value theorem gives, for every $x>0$, a unique
positive $u$ such that $x=b(u)$.  For every $p>0$, define
\begin{equation}\label{eq:q6-Wp}
 W_p(x):=\frac{m(u)u^{-p-1}}{b'(u)}
 =\frac14u^{-p-3}(1+4u^2)^{3/4}.
\end{equation}
The uniqueness of $u$ shows that $W_p$ is well defined on $(0,\infty)$.
Changing variables from $u$ to $x=b(u)$ gives
\begin{equation}\label{eq:q6-phase-change-variable}
 \int_0^\infty m(u)[\cos(\delta b(u))-\cos b(u)]u^{-p-1}\dd u
 =\int_0^\infty W_p(x)[\cos(\delta x)-\cos x]\dd x.
\end{equation}
We now show that $W_p$ is completely monotone.  Let $v=1/(2u)$.  Then
\[
 W_p(x)=2^{p+1}v^{p+3/2}(1+v^2)^{3/4}.
\]
Since $-\dd v/\dd x=2v^2(1+v^2)$, for all $\alpha,\beta>0$ the chain rule gives
\begin{align}
 -\frac{\dd}{\dd x}[v^\alpha(1+v^2)^\beta]
 &=2v^2(1+v^2)\frac{\dd}{\dd v}[v^\alpha(1+v^2)^\beta]\notag\\
 &=2\alpha v^{\alpha+1}(1+v^2)^{\beta+1}
   +4\beta v^{\alpha+3}(1+v^2)^\beta.
 \label{eq:q6-W-derivative}
\end{align}
Each term on the right is a positive constant times a positive power of
$v$ and a positive power of $1+v^2$.  Applying
\eqref{eq:q6-W-derivative} to every term preserves this form and the
positivity of every coefficient.  Starting from
$W_p(x)=2^{p+1}v^{p+3/2}(1+v^2)^{3/4}$ and repeating this argument $k$
times gives $(-1)^kW_p^{(k)}(x)>0$ for every non-negative integer $k$.
The calculation uses only $p>0$, so $W_p$ is completely monotone for every
$p>0$.

We use Bernstein's Laplace-mixture
theorem~\cite[Theorem~1.4]{SchillingEtAl}: if a smooth function $W$ on
$(0,\infty)$ satisfies $(-1)^kW^{(k)}\ge0$ for every $k\ge0$, then
$W(x)=\int_0^\infty\exp(-tx)\rho(\dd t)$ for a positive measure $\rho$,
which need not be finite.
Applying it to $W_p$ gives a nonzero positive measure $\rho_p$ such that
\begin{equation}\label{eq:q6-Bernstein-representation}
 W_p(x)=\int_0^\infty \exp(-tx)\rho_p(\dd t).
\end{equation}
Equation~\eqref{eq:q6-Wp} and $b(u)\to\infty$ show that
$W_p(x)\to0$ as $x\to\infty$.  Since
$\exp(-tx)\le\exp(-t)$ for $x\ge1$ and
$\int_0^\infty\exp(-t)\rho_p(\dd t)=W_p(1)<\infty$, the dominated
convergence theorem gives
$\rho_p(\{0\})=\lim_{x\to\infty}W_p(x)=0$.  The measure $\rho_p$ is nonzero,
so $\rho_p((0,\infty))>0$.  Moreover,
$|\cos(\delta x)-\cos x|\le C\min\{x^2,1\}$.  Equations
\eqref{eq:q6-phase-absolute} and \eqref{eq:q6-phase-change-variable} show
that its product with $W_p$ is integrable.  Tonelli's theorem gives
\[
 \int_0^\infty\int_0^\infty \exp(-tx)
 |\cos(\delta x)-\cos x|\dd x\,\rho_p(\dd t)
 =\int_0^\infty W_p(x)|\cos(\delta x)-\cos x|\dd x<\infty.
\]
Hence Fubini is justified.
Therefore
\begin{align*}
 \int_0^\infty W_p(x)[\cos(\delta x)-\cos x]\dd x
 &=\int_0^\infty\left[\int_0^\infty \exp(-tx)
 (\cos(\delta x)-\cos x)\dd x\right]\rho_p(\dd t)\\
 &=\int_0^\infty\left(\frac{t}{t^2+\delta^2}
 -\frac{t}{t^2+1}\right)\rho_p(\dd t)>0.
\end{align*}
Here the strict inequality uses $0\le\delta<1$ and
$\rho_p((0,\infty))>0$.  This proves
\eqref{eq:q6-phase-positive}.

\medskip
\noindent\emph{From the integral to the moments.}
For $4<p<6$, let
$\psi(s)=\cos s-1+\frac{1}{2}s^2-\frac{1}{24}s^4$.
The function $\psi(s)s^{-p-1}$ is $O(s^{5-p})$ at zero and
$O(s^{3-p})$ at infinity, so its integral over $(0,\infty)$ is finite.
Both $\psi$ and $\psi'$ vanish at zero.  Since
$1-\cos s<\frac{1}{2}s^2$, we have
$\psi''(s)=1-\cos s-\frac{1}{2}s^2<0$ for $s>0$.
It follows first that $\psi'(s)<0$ and then that $\psi(s)<0$ for $s>0$.
Therefore
\[
 c_p:=\left[\int_0^\infty\psi(s)s^{-p-1}\dd s\right]^{-1}<0.
\]
Since $\psi$ is even, for every $y\ne0$ the substitution $s=u|y|$ gives
\begin{equation}\label{eq:q6-psi-pointwise}
 c_p\int_0^\infty\psi(uy)u^{-p-1}\dd u
 =c_p|y|^p\int_0^\infty\psi(s)s^{-p-1}\dd s=|y|^p.
\end{equation}
For $y=0$, both sides are zero.  Since $\psi$ is nonpositive, for every
random variable $Y$ with
$\E|Y|^p<\infty$, Tonelli's theorem and \eqref{eq:q6-psi-pointwise}
give
\begin{equation}\label{eq:q6-psi-absolute-integrability}
 \int_0^\infty\E|\psi(uY)|u^{-p-1}\dd u
 =\E\left[\int_0^\infty|\psi(uY)|u^{-p-1}\dd u\right]
 =-\frac{1}{c_p}\E|Y|^p.
\end{equation}
Since $c_p<0$ and $\E|Y|^p<\infty$, the right-hand side of
\eqref{eq:q6-psi-absolute-integrability} is finite.  Fubini's theorem
therefore gives the following absolute-moment representation:
\begin{equation}\label{eq:q6-absolute-moment-representation}
 \E|Y|^p=c_p\int_0^\infty
 [\E\cos(uY)-1+\frac{1}{2}u^2\E Y^2-\frac{1}{24}u^4\E Y^4]
 u^{-p-1}\dd u.
\end{equation}
By \eqref{eq:q6-low-moments},
$\E Z_\delta^2=\E Z_1^2=2$ and
$\E Z_\delta^4=\E Z_1^4=60$.  Applying
\eqref{eq:q6-absolute-moment-representation} first to $Z_\delta$ and then
to $Z_1$, and subtracting the two equalities, gives
\begin{align}
 {}&\E|Z_\delta|^p-\E|Z_1|^p\notag\\
 ={}&c_p\int_0^\infty\Bigl[\E\cos(uZ_\delta)-\E\cos(uZ_1)
   +\frac{1}{2}u^2(\E Z_\delta^2-\E Z_1^2)
   -\frac{1}{24}u^4(\E Z_\delta^4-\E Z_1^4)\Bigr]u^{-p-1}\dd u\notag\\
 ={}&c_p\int_0^\infty
   [\E\cos(uZ_\delta)-\E\cos(uZ_1)]u^{-p-1}\dd u\notag\\
 ={}&c_p\int_0^\infty m(u)[\cos(\delta b(u))-\cos b(u)]
   u^{-p-1}\dd u.
 \label{eq:q6-phase-moment-difference}
\end{align}
Here the last equality follows from \eqref{eq:q6-trig-transforms}.  Since
$Z_1\stackrel d=g^2-1$, the left-hand side of
\eqref{eq:q6-phase-moment-difference} is
$\E|Z_\delta|^p-\E|g^2-1|^p$.
For $0\le\delta<1$, the sign of $c_p$ and
\eqref{eq:q6-phase-positive} make the right-hand side strictly negative.
For $\delta=1$, the two moments are equal.  This proves part~(2), including
its equality condition.
\end{proof}

The comparison in \cref{thm:gamma-phase-comparison} can be extended to every
$p\ge4$.
Signed absolute moments change sign when $\delta$ is replaced by $-\delta$.
Thus it suffices to take $0\le\delta\le1$ throughout the
signed-moment argument below.

\begin{theorem}
\label{thm:q6-scalar-all-high}
For every $p\ge4$ and every $0\le\delta\le1$,
\begin{equation}\label{eq:q6-scalar-all-high}
 \E|Z_\delta|^p\le\E|Z_1|^p=\E|g^2-1|^p.
\end{equation}
For every $p>4$ and every $0\le\delta<1$, the inequality is strict.
\end{theorem}

\begin{proof}
\noindent\emph{Absolute moments of order at most two.}
The argument proving \eqref{eq:q6-phase-positive} only requires $0<p<6$.
Hence, for $0<p<2$, its integral is positive when $0\le\delta<1$ and is
zero when $\delta=1$.  Moreover,
$(1-\cos s)s^{-p-1}$ is $O(s^{1-p})$ at zero and $O(s^{-p-1})$ at
infinity.  It is nonnegative for $s>0$ and strictly positive on
$(0,2\pi)$.  Its integral is therefore finite and strictly positive.  For
$y\ne0$, the substitution $s=u|y|$ gives
\[
 \int_0^\infty[1-\cos(uy)]u^{-p-1}\dd u
 =|y|^p\int_0^\infty(1-\cos s)s^{-p-1}\dd s.
\]
The identity also holds for $y=0$.  Tonelli's theorem now gives
\begin{equation}\label{eq:q6-low-absolute-representation}
 \E|Y|^p=
 \frac{\displaystyle\int_0^\infty
 [1-\E\cos(uY)]u^{-p-1}\dd u}
 {\displaystyle\int_0^\infty(1-\cos s)s^{-p-1}\dd s}.
\end{equation}
Applying \eqref{eq:q6-low-absolute-representation} to $Z_\delta$ and
$Z_1$, and subtracting the two equalities, gives
\begin{align*}
 \E|Z_\delta|^p-\E|Z_1|^p
 &=\frac{\displaystyle\int_0^\infty
 [\E\cos(uZ_1)-\E\cos(uZ_\delta)]u^{-p-1}\dd u}
 {\displaystyle\int_0^\infty(1-\cos s)s^{-p-1}\dd s}\\
 &=-\frac{\displaystyle\int_0^\infty m(u)
 [\cos(\delta b(u))-\cos b(u)]u^{-p-1}\dd u}
 {\displaystyle\int_0^\infty(1-\cos s)s^{-p-1}\dd s}\le0.
\end{align*}
Here the second equality follows from \eqref{eq:q6-trig-transforms}, and
the inequality follows from \eqref{eq:q6-phase-positive}.  Thus,
for all $0\le p\le2$,
\begin{equation}\label{eq:q6-absolute-low-base}
 \E|Z_\delta|^p\le\E|Z_1|^p,
\end{equation}
where the endpoints follow from $\E|Z_\delta|^0=1$ and
$\E|Z_\delta|^2=2$.

\medskip
\noindent\emph{Signed moments of order from one to three.}
For all $1\le q\le3$, we prove the signed comparison
\begin{equation}\label{eq:q6-signed-low-base}
 0\le\delta\E[Z_\delta|Z_\delta|^{q-1}]
 \le\E[Z_1|Z_1|^{q-1}].
\end{equation}
For $1<q<3$, let
$D_q=\int_0^\infty(\sin s-s)s^{-q-1}\dd s$.  The integrand is
$O(s^{2-q})$ at zero and $O(s^{-q})$ at infinity, so the integral is
finite.  Since $\sin s-s<0$ for $s>0$, we have $D_q<0$.  For $x>0$,
the substitution $s=ux$ gives
\[
 \int_0^\infty(\sin(ux)-ux)u^{-q-1}\dd u
 =x^qD_q.
\]
The integral on the left is an odd function of $x$.  Since
$x|x|^{q-1}=x^q$ for $x>0$, we obtain, for every $x\in\R$,
\[
 x|x|^{q-1}=D_q^{-1}\int_0^\infty
 (\sin(ux)-ux)u^{-q-1}\dd u.
\]
For $Y=Z_\delta$ or $Y=Z_1$, the substitution $s=u|Y|$ and Tonelli's
theorem give
\[
 \int_0^\infty\E|\sin(uY)-uY|u^{-q-1}\dd u
 =\E|Y|^q\int_0^\infty|\sin s-s|s^{-q-1}\dd s<\infty.
\]
The right-hand side is finite because the Gamma-difference variables have
moments of every order.  Therefore taking expectations in the identity for
$x=Z_\delta$ and $x=Z_1$ is justified by Fubini's theorem.  Since
$\E Z_\delta=\E Z_1=0$, \eqref{eq:q6-trig-transforms} gives
\begin{align}
 \E[Z_\delta|Z_\delta|^{q-1}]
 &=D_q^{-1}\int_0^\infty
   \E[\sin(uZ_\delta)-uZ_\delta]u^{-q-1}\dd u\notag\\
 &=D_q^{-1}\int_0^\infty
   \E\sin(uZ_\delta)u^{-q-1}\dd u\notag\\
 &=-D_q^{-1}\int_0^\infty
   m(u)\sin(\delta b(u))u^{-q-1}\dd u.
   \label{eq:q6-signed-low-zdelta}\\
 \E[Z_1|Z_1|^{q-1}]
 &=D_q^{-1}\int_0^\infty
   \E[\sin(uZ_1)-uZ_1]u^{-q-1}\dd u\notag\\
 &=D_q^{-1}\int_0^\infty
   \E\sin(uZ_1)u^{-q-1}\dd u\notag\\
 &=-D_q^{-1}\int_0^\infty
   m(u)\sin b(u)u^{-q-1}\dd u.
   \label{eq:q6-signed-low-zone}
\end{align}
Thus,
\begin{equation}\label{eq:q6-signed-low-difference}
 \E[Z_1|Z_1|^{q-1}]
 -\delta\E[Z_\delta|Z_\delta|^{q-1}]
 =-D_q^{-1}\int_0^\infty m(u)
 [\sin b(u)-\delta\sin(\delta b(u))]u^{-q-1}\dd u.
\end{equation}
Since $-D_q^{-1}>0$, \eqref{eq:q6-signed-low-difference} shows that the upper bound in
\eqref{eq:q6-signed-low-base} is equivalent to
\[
 \int_0^\infty m(u)
 [\sin b(u)-\delta\sin(\delta b(u))]u^{-q-1}\dd u\ge0.
\]
For every $t>0$, direct integration gives
\begin{equation}\label{eq:q6-sine-laplace-transform}
 \int_0^\infty \exp(-tx)[\sin x-\delta\sin(\delta x)]\dd x
 =\frac1{t^2+1}-\frac{\delta^2}{t^2+\delta^2}
 =\frac{t^2(1-\delta^2)}{(t^2+1)(t^2+\delta^2)}\ge0.
\end{equation}
The proof of complete monotonicity for $W_p$ uses only $p>0$.  Since $q>0$,
$W_q$ is completely monotone.  Let $\rho_q$ be the positive measure
representing $W_q$ in \eqref{eq:q6-Bernstein-representation}.  To justify
Fubini's theorem, first note that
$m(u)|\sin b(u)-\delta\sin(\delta b(u))|u^{-q-1}$ is
$O(u^{2-q})$ at zero and $O(u^{-q-3/2})$ at infinity, hence integrable for
$1<q<3$.  Equation~\eqref{eq:q6-Wp}, the substitution $x=b(u)$,
\eqref{eq:q6-Bernstein-representation}, and Tonelli's theorem give
\begin{align*}
 {}&\int_0^\infty m(u)
 |\sin b(u)-\delta\sin(\delta b(u))|u^{-q-1}\dd u\\
 ={}&\int_0^\infty W_q(x)|\sin x-\delta\sin(\delta x)|\dd x\\
 ={}&\int_0^\infty\int_0^\infty \exp(-tx)
 |\sin x-\delta\sin(\delta x)|\dd x\,\rho_q(\dd t)<\infty.
\end{align*}
Fubini's theorem therefore applies.  Equations~\eqref{eq:q6-Wp},
\eqref{eq:q6-Bernstein-representation}, and
\eqref{eq:q6-sine-laplace-transform} give
\begin{align*}
 {}&\int_0^\infty m(u)
 [\sin b(u)-\delta\sin(\delta b(u))]u^{-q-1}\dd u\\
 ={}&\int_0^\infty W_q(x)[\sin x-\delta\sin(\delta x)]\dd x\\
 ={}&\int_0^\infty\left[\int_0^\infty \exp(-tx)
 [\sin x-\delta\sin(\delta x)]\dd x\right]\rho_q(\dd t)\\
 ={}&\int_0^\infty
 \frac{t^2(1-\delta^2)}{(t^2+1)(t^2+\delta^2)}\rho_q(\dd t)\ge0.
\end{align*}
This proves the second inequality
in \eqref{eq:q6-signed-low-base}.  To prove the first inequality, we also exchange
the Laplace mixture and integration for $\sin(\delta x)$.  Since
$\sin(\delta b(u))=O(u^3)$ at zero and is bounded at infinity,
\[
 \int_0^\infty m(u)|\sin(\delta b(u))|u^{-q-1}\dd u<\infty.
\]
When $\delta=0$, \eqref{eq:q6-signed-low-zdelta} gives
$\E[Z_\delta|Z_\delta|^{q-1}]=0$.  When $0<\delta\le1$,
Tonelli's theorem applied to the absolute value and then Fubini's theorem
give
\begin{align*}
 \int_0^\infty m(u)\sin(\delta b(u))u^{-q-1}\dd u
 &=\int_0^\infty W_q(x)\sin(\delta x)\dd x\\
 &=\int_0^\infty\left[\int_0^\infty
   \exp(-tx)\sin(\delta x)\dd x\right]\rho_q(\dd t)\\
 &=\int_0^\infty\frac{\delta}{t^2+\delta^2}\rho_q(\dd t)\ge0.
\end{align*}
Since $-D_q^{-1}>0$, \eqref{eq:q6-signed-low-zdelta} shows that
$\E[Z_\delta|Z_\delta|^{q-1}]\ge0$.
For $q=1$, all three terms in \eqref{eq:q6-signed-low-base} are zero.  For
$q=3$, the two inequalities reduce to $0\le8\delta^2\le8$.  This proves
\eqref{eq:q6-signed-low-base} at both endpoints.

\medskip
\noindent\emph{Signed moments of order from five to seven.}
For all $5\le q\le7$, we next prove
\begin{equation}\label{eq:q6-signed-five-seven}
 \delta\E[Z_\delta|Z_\delta|^{q-1}]
 \le\E[Z_1|Z_1|^{q-1}].
\end{equation}
Let $5<q<7$ and let
\[
 R_\delta(x)=\sin x-\delta\sin(\delta x)
 -(1-\delta^2)x+\frac{1}{6}(1-\delta^4)x^3.
\]
Taylor's formula gives
$\sin s-s+\frac{1}{6}s^3-\frac{1}{120}s^5=O(s^7)$ as $s\downarrow0$.
For $s\ge1$, $|\sin s|\le1$, and hence
\[
 \left|\sin s-s+\frac{1}{6}s^3-\frac{1}{120}s^5\right|
 \le1+s+\frac{1}{6}s^3+\frac{1}{120}s^5\le3s^5.
\]
Consequently,
$\sin s-s+\frac{1}{6}s^3-\frac{1}{120}s^5=O(s^5)$ as $s\to\infty$.
After multiplication by $s^{-q-1}$, these bounds are integrable because
$5<q<7$.
\begin{samepage}
Let $J_q$ denote the following integral.  Taylor's formula with
integral remainder and $\sin t<t$ for $t>0$ give
\begin{align*}
 J_q
 &=\int_0^\infty
 \left(\sin s-s+\frac{1}{6}s^3-\frac{1}{120}s^5\right)s^{-q-1}\dd s\\
 &=-\frac{1}{6}\int_0^\infty
 \left[\int_0^s(s-t)^3(t-\sin t)\dd t\right]s^{-q-1}\dd s<0.
\end{align*}
\end{samepage}
For $x>0$, the substitution $s=ux$ gives
\[
 \int_0^\infty
 \left(\sin(ux)-ux+\frac{1}{6}(ux)^3-\frac{1}{120}(ux)^5\right)
 u^{-q-1}\dd u=x^qJ_q.
\]
The integral on the left is an odd function of $x$.  Since
$x|x|^{q-1}=x^q$ for $x>0$, we obtain, for every $x\in\R$,
\begin{equation}\label{eq:q6-finite-part-pointwise}
 x|x|^{q-1}=
 \frac1{J_q}\int_0^\infty
 \left(\sin(ux)-ux+\frac{1}{6}(ux)^3-\frac{1}{120}(ux)^5\right)
 u^{-q-1}\dd u.
\end{equation}
After multiplication by $s^{-q-1}$, the absolute value of the integrand is
$O(s^{6-q})$ at zero and $O(s^{4-q})$ at infinity.  Hence its integral is
finite for $5<q<7$.
The substitution $s=u|x|$ gives, for every $x\in\R$,
\begin{align*}
 {}&\int_0^\infty
 \left|\sin(ux)-ux+\frac{1}{6}(ux)^3-\frac{1}{120}(ux)^5\right|
 u^{-q-1}\dd u\\
 ={}&|x|^q\int_0^\infty
 \left|\sin s-s+\frac{1}{6}s^3-\frac{1}{120}s^5\right|
 s^{-q-1}\dd s.
\end{align*}
Therefore, for $Y=g^2-1$ or $Y=Z_\delta$, Tonelli's theorem gives
\begin{align}
 {}&\int_0^\infty\E\left|
 \sin(uY)-uY+\frac{1}{6}(uY)^3-\frac{1}{120}(uY)^5
 \right|u^{-q-1}\dd u\notag\\
 ={}&\E|Y|^q\int_0^\infty
 \left|\sin s-s+\frac{1}{6}s^3-\frac{1}{120}s^5\right|
 s^{-q-1}\dd s<\infty.
 \label{eq:q6-finite-part-fubini}
\end{align}
By \eqref{eq:q6-finite-part-fubini}, Fubini's theorem applies when taking
expectations in \eqref{eq:q6-finite-part-pointwise}.  Since
$\E Z_\delta=0$, $\E Z_\delta^3=8\delta$, and
$\E Z_\delta^5=544\delta$, \eqref{eq:q6-finite-part-pointwise} gives
\begin{align*}
 \E[Z_\delta|Z_\delta|^{q-1}]
 &=\frac1{J_q}\int_0^\infty
 \left[\E\sin(uZ_\delta)-u\E Z_\delta
 +\frac{1}{6}u^3\E Z_\delta^3-\frac{1}{120}u^5\E Z_\delta^5\right]
 u^{-q-1}\dd u\\
 &=\frac1{J_q}\int_0^\infty
 \left[\E\sin(uZ_\delta)
 +\frac{4}{3}\delta u^3-\frac{68}{15}\delta u^5\right]
 u^{-q-1}\dd u\\
 &=\frac1{J_q}\int_0^\infty
 \left[-m(u)\sin(\delta b(u))
 +\frac{4}{3}\delta u^3-\frac{68}{15}\delta u^5\right]
 u^{-q-1}\dd u.
\end{align*}
Here the last equality follows from \eqref{eq:q6-trig-transforms}.  Thus,
\begin{align}
 {}&\E[Z_1|Z_1|^{q-1}]
 -\delta\E[Z_\delta|Z_\delta|^{q-1}]\notag\\
 ={}&\frac1{J_q}\int_0^\infty
 \Bigl[-m(u)\bigl(\sin b(u)-\delta\sin(\delta b(u))\bigr)
 +(1-\delta^2)\left(\frac43u^3-\frac{68}{15}u^5\right)\Bigr]
 u^{-q-1}\dd u.
 \label{eq:q6-signed-finite-part}
\end{align}
The numerator in \eqref{eq:q6-signed-finite-part} is of order $O(u^7)$
at zero, so the weighted integrand is $O(u^{6-q})$.  At infinity the
polynomial term gives $O(u^{4-q})$.  Since $5<q<7$, both bounds are
integrable, and the integral is absolutely convergent.

For $c\ge0$ and $t>0$, integration gives
\[
 \int_0^\infty \exp(-tx)\sin(cx)\dd x=\frac{c}{t^2+c^2},\qquad
 \int_0^\infty x\exp(-tx)\dd x=\frac1{t^2},\qquad
 \int_0^\infty x^3\exp(-tx)\dd x=\frac6{t^4}
\]
and hence
\begin{align}
 \int_0^\infty \exp(-tx)R_\delta(x)\dd x
 &=\frac1{t^2+1}-\frac{\delta^2}{t^2+\delta^2}
   -\frac{1-\delta^2}{t^2}+\frac{1-\delta^4}{t^4}\notag\\
 &=\frac{(1-\delta^2)
 [(1+\delta^2+\delta^4)t^2+\delta^2(1+\delta^2)]}
 {t^4(t^2+1)(t^2+\delta^2)}\ge0.
 \label{eq:q6-Rdelta-laplace}
\end{align}
Since $R_\delta(x)=O(x^5)$ and $b(u)/u^3\longrightarrow4/3$ at zero, while
$R_\delta(x)=O(x^3)$ and $b(u)/u\longrightarrow1$ at infinity,
$m(u)|R_\delta(b(u))|u^{-q-1}$ is $O(u^{14-q})$ at zero and
$O(u^{3/2-q})$ at infinity.  Both bounds are integrable for $5<q<7$, so
\[
 \int_0^\infty m(u)|R_\delta(b(u))|u^{-q-1}\dd u<\infty.
\]
The proof of complete monotonicity for $W_p$ uses only $p>0$.  Hence $W_q$
is completely monotone.  Let $\rho_q$ be the positive measure representing
$W_q$ in \eqref{eq:q6-Bernstein-representation}.
Equation~\eqref{eq:q6-Bernstein-representation} and Tonelli's
theorem give
\begin{align*}
 {}&\int_0^\infty\int_0^\infty
   \exp(-tx)|R_\delta(x)|\dd x\,\rho_q(\dd t)\\
 ={}&\int_0^\infty W_q(x)|R_\delta(x)|\dd x\\
 ={}&\int_0^\infty m(u)|R_\delta(b(u))|u^{-q-1}\dd u<\infty.
\end{align*}
Fubini's theorem now gives
\begin{align}
 \int_0^\infty m(u)R_\delta(b(u))u^{-q-1}\dd u
 &=\int_0^\infty W_q(x)R_\delta(x)\dd x\notag\\
 &=\int_0^\infty\left[\int_0^\infty
   \exp(-tx)R_\delta(x)\dd x\right]\rho_q(\dd t)\ge0.
 \label{eq:q6-Rdelta-positive}
\end{align}
Here the inequality follows from \eqref{eq:q6-Rdelta-laplace}.
The numerator in \eqref{eq:q6-signed-finite-part} equals
\[
 -m(u)R_\delta(b(u))+(1-\delta^2)
 \left[\frac43u^3-\frac{68}{15}u^5-m(u)b(u)
       +\frac{1}{6}(1+\delta^2)m(u)b(u)^3\right].
\]
Equation~\eqref{eq:q6-Rdelta-positive} gives
\[
 -\int_0^\infty m(u)R_\delta(b(u))u^{-q-1}\dd u\le0.
\]
It remains to prove, for all $u>0$, that
\begin{equation}\label{eq:q6-finite-part-remainder-bound}
 \frac43u^3-\frac{68}{15}u^5-m(u)b(u)
 +\frac{1}{6}(1+\delta^2)m(u)b(u)^3\le0.
\end{equation}
Since $0\le\delta\le1$, we may replace $(1+\delta^2)/6$ by $1/3$.
Because $(1+4s^2)^{-1}\ge1-4s^2$, we have
$4s^2-16s^4\le 4s^2/(1+4s^2)\le4s^2$.
Integration from $0$ to $u$ gives
\[
 \frac43u^3-\frac{16}{5}u^5
 =\int_0^u(4s^2-16s^4)\dd s
 \le b(u)
 \le\int_0^u4s^2\dd s
 =\frac43u^3.
\]
Moreover, $b(u)=u-\frac12\arctan(2u)\le u$.  Bernoulli's inequality
applied to $m(u)=(1+4u^2)^{-1/4}$ gives $m(u)\ge1-u^2$.  If
$u^2\le25/68$, then both $1-u^2$ and
$\frac43u^3-\frac{16}{5}u^5$ are nonnegative, and hence
\[
 m(u)b(u)\ge
 (1-u^2)\left(\frac43u^3-\frac{16}{5}u^5\right)
 =\frac43u^3-\frac{68}{15}u^5+\frac{16}{5}u^7.
\]
Together with $m(u)b(u)^3\le(4u^3/3)^3$, this gives
\[
 \frac43u^3-\frac{68}{15}u^5-m(u)b(u)
 +\frac{1}{6}(1+\delta^2)m(u)b(u)^3
 \le-\frac{16}{5}u^7+\frac{64}{81}u^9<0.
\]
If $u^2\ge25/68$, discard the negative term $-m(u)b(u)$ and use
$m(u)\le1$, $b(u)\le u$ to obtain
\[
 \frac43u^3-\frac{68}{15}u^5-m(u)b(u)
 +\frac{1}{6}(1+\delta^2)m(u)b(u)^3
 \le\frac43u^3-\frac{68}{15}u^5+\frac13u^3\le0.
\]
The two cases prove \eqref{eq:q6-finite-part-remainder-bound}.  Combining
\eqref{eq:q6-finite-part-remainder-bound} with
\eqref{eq:q6-Rdelta-positive} gives
\begin{align*}
 &\int_0^\infty
 \left[-m(u)\bigl(\sin b(u)-\delta\sin(\delta b(u))\bigr)
 +(1-\delta^2)\left(\frac43u^3-\frac{68}{15}u^5\right)\right]
 u^{-q-1}\dd u\le0.
\end{align*}
Thus the integral in \eqref{eq:q6-signed-finite-part} is
nonpositive.  Since $J_q<0$,
\eqref{eq:q6-signed-five-seven} follows for $5<q<7$.  For $5\le q\le7$,
$|Z_\delta|^q\le1+|Z_\delta|^8$.
Moreover, $|Z_\delta|\le A+B+1$ and
$A+B\sim\operatorname{Gamma}(\frac12,\frac12)$, independently of
$\delta$.  Hence
$q\mapsto\E[Z_\delta|Z_\delta|^{q-1}]$ is continuous on $[5,7]$ by
the dominated convergence theorem.  Letting $q\downarrow5$ and $q\uparrow7$ gives the
cases $q=5$ and $q=7$.

\medskip
\noindent\emph{Recurrences and induction.}
It remains to extend the base inequalities to all higher orders.  A Stein
identity for the difference of two independent Gamma variables with
possibly different shapes and rates is given in
\mbox{\cite[Proposition~1]{Forrester}}.
In that proposition, take the two shapes to be $\frac{1}{4}(1+\delta)$ and
$\frac{1}{4}(1-\delta)$, and take both rates to be $1/2$.  With these
parameters, the cited identity gives
\begin{equation}\label{eq:q6-gamma-Stein}
 \E\left[(A-B)f''(A-B)+\frac12f'(A-B)
 +\left(\frac\delta4-\frac{1}{4}(A-B)\right)f(A-B)\right]=0.
\end{equation}
Applying \eqref{eq:q6-gamma-Stein} to the function
$x\mapsto f(x-\delta)$ and using $A-B=Z_\delta+\delta$ gives
\[
 \E\left[(Z_\delta+\delta)f''(Z_\delta)+\frac12f'(Z_\delta)
 -\frac14Z_\delta f(Z_\delta)\right]=0.
\]
Rearranging gives
\begin{equation}\label{eq:q6-Z-Stein}
 \E[Z_\delta f(Z_\delta)]
 =\E[2f'(Z_\delta)+4(Z_\delta+\delta)f''(Z_\delta)].
\end{equation}
The cited proposition assumes positive shapes, so we first take
$0\le\delta<1$.  For $p>3$, the function $f(x)=x|x|^{p-2}$ is $C^2$, and all terms in
\eqref{eq:q6-Z-Stein} have finite expectations.  Its derivatives are
$f'(x)=(p-1)|x|^{p-2}$ and
$f''(x)=(p-1)(p-2)x|x|^{p-4}$.
Substitution in \eqref{eq:q6-Z-Stein} gives
\begin{align}
 \E|Z_\delta|^p
 &=2(p-1)\E|Z_\delta|^{p-2}
   +4(p-1)(p-2)
    \left[\E|Z_\delta|^{p-2}
    +\delta\E[Z_\delta|Z_\delta|^{p-4}]\right]\notag\\
 &=2(p-1)(2p-3)\E|Z_\delta|^{p-2}
   +4\delta(p-1)(p-2)\E[Z_\delta|Z_\delta|^{p-4}].
   \label{eq:q6-absolute-recurrence}
\end{align}
For $q>3$, the function $f(x)=|x|^{q-1}$ is also $C^2$, and all terms in
\eqref{eq:q6-Z-Stein} have finite expectations.  Its derivatives are
$f'(x)=(q-1)x|x|^{q-3}$ and
$f''(x)=(q-1)(q-2)|x|^{q-3}$.
Another substitution gives
\begin{align}
 \E[Z_\delta|Z_\delta|^{q-1}]
 &=2(q-1)\E[Z_\delta|Z_\delta|^{q-3}]
   +4(q-1)(q-2)
    \left[\E[Z_\delta|Z_\delta|^{q-3}]
    +\delta\E|Z_\delta|^{q-3}\right]\notag\\
 &=2(q-1)(2q-3)\E[Z_\delta|Z_\delta|^{q-3}]
   +4\delta(q-1)(q-2)\E|Z_\delta|^{q-3}.
 \label{eq:q6-signed-recurrence}
\end{align}
Both recurrences also hold at $\delta=1$.  To see this, let
$\delta\uparrow1$.  Equation~\eqref{eq:q6-Zdelta-mgf} shows that the MGF of
$Z_\delta$ converges on $(-1/2,1/2)$ to the MGF of $Z_1$, so
$Z_\delta$ converges to $Z_1$ in distribution.  Moreover,
$|Z_\delta|\le A+B+1$ and
$A+B\sim\operatorname{Gamma}(\frac12,\frac12)$ for every $\delta$.
Because this Gamma variable has moments of every order, all terms in the two
recurrences are uniformly integrable.  Their expectations therefore
converge, which proves the recurrences at $\delta=1$.
We first record two consequences of the recurrences.  Suppose first that
$q>3$ and that
$\delta\E[Z_\delta|Z_\delta|^{q-3}]
\le\E[Z_1|Z_1|^{q-3}]$ and
$\E|Z_\delta|^{q-3}\le\E|Z_1|^{q-3}$.  The positive coefficients in
\eqref{eq:q6-signed-recurrence} and $0\le\delta\le1$ give
\begin{align}
 \delta\E[Z_\delta|Z_\delta|^{q-1}]
 &=2(q-1)(2q-3)\delta\E[Z_\delta|Z_\delta|^{q-3}]
   +4(q-1)(q-2)\delta^2\E|Z_\delta|^{q-3}\notag\\
 &\le2(q-1)(2q-3)\E[Z_1|Z_1|^{q-3}]
   +4(q-1)(q-2)\E|Z_1|^{q-3}\notag\\
 &=\E[Z_1|Z_1|^{q-1}].
 \label{eq:q6-signed-induction-step}
\end{align}
Next suppose that $p>3$,
$\E|Z_\delta|^{p-2}\le\E|Z_1|^{p-2}$, and
$\delta\E[Z_\delta|Z_\delta|^{p-4}]
\le\E[Z_1|Z_1|^{p-4}]$.  Then
\begin{align}
 \E|Z_\delta|^p
 &=2(p-1)(2p-3)\E|Z_\delta|^{p-2}
   +4(p-1)(p-2)\delta\E[Z_\delta|Z_\delta|^{p-4}]\notag\\
 &\le2(p-1)(2p-3)\E|Z_1|^{p-2}
   +4(p-1)(p-2)\E[Z_1|Z_1|^{p-4}]\notag\\
 &=\E|Z_1|^p.
 \label{eq:q6-absolute-induction-step}
\end{align}

For $3<q\le5$, one has $1<q-2\le3$ and $0<q-3\le2$.
Consequently, \eqref{eq:q6-signed-low-base} and
\eqref{eq:q6-absolute-low-base} give the two assumptions used in
\eqref{eq:q6-signed-induction-step}.  Thus the signed
upper comparison holds for $3\le q\le5$.  Theorem~\ref{thm:gamma-phase-comparison}
gives the moment comparison for $4\le p\le6$.  For $6<p\le8$, one has
$4<p-2\le6$ and $3<p-3\le5$.  Hence the moment comparison on $[4,6]$
and the signed comparison on $[3,5]$ give the two assumptions used in
\eqref{eq:q6-absolute-induction-step}.  Together
with \eqref{eq:q6-signed-five-seven}, the base intervals are therefore $[4,8]$
for the moment comparison and $[3,7]$ for the signed upper comparison.
For $6<p\le8$, the inequality is strict when $\delta<1$, because
$4<p-2\le6$, the coefficient of $\E|Z_\delta|^{p-2}$ in
\eqref{eq:q6-absolute-recurrence} is positive, and
$\E|Z_\delta|^{p-2}<\E|Z_1|^{p-2}$.

Suppose for some integer $k\ge0$ that these comparisons are known on
$[4,8+2k]$ and $[3,7+2k]$, respectively.  If
$7+2k\le q\le9+2k$, then $q-2\in[5+2k,7+2k]$ and
$q-3\in[4+2k,6+2k]$.  Equation~\eqref{eq:q6-signed-induction-step} extends
the signed upper comparison to $[3,9+2k]$.  If
$8+2k\le p\le10+2k$, then
$p-2\in[6+2k,8+2k]$ and $p-3\in[5+2k,7+2k]$.  Equation
\eqref{eq:q6-absolute-induction-step} extends the moment comparison to
$[4,10+2k]$.  Induction on $k$ covers every $p\ge4$.
For strictness, suppose that $\delta<1$.  In each new interval for $p$, one
has $p-2>4$, so the first inequality in
\eqref{eq:q6-absolute-induction-step} is strict by the induction hypothesis.
Starting from the strict inequality on $4<p\le8$, induction therefore gives
strictness for every $p>4$.  Thus
\eqref{eq:q6-scalar-all-high} is strict when $\delta<1$.
\end{proof}

\begin{proof}[Proof of \cref{thm:q6-main}]
Let $M$ and $G$ be as in the theorem, and let
$\delta=\operatorname{tr}(M^3)/\|M\|_{\mathrm F}^3$.
Since $Z_{-\delta}\stackrel d=-Z_\delta$, one has
$\E|Z_\delta|^p=\E|Z_{|\delta|}|^p$.
Corollary~\ref{cor:q6-matrix-envelope} and
Theorem~\ref{thm:q6-scalar-all-high} therefore give
\[
 \E\left|G^{\mathsf T}MG-\operatorname{tr}M\right|^p
 \le \|M\|_{\mathrm F}^p\E|Z_\delta|^p
 =\|M\|_{\mathrm F}^p\E|Z_{|\delta|}|^p
 \le \|M\|_{\mathrm F}^p\E|g^2-1|^p.
\]
Taking $p$th roots proves the stated matrix inequality.

At $p=4$, \eqref{eq:q6-diagonal-reduction} and
Lemma~\ref{lem:q6-endpoints} give the equality cases.  Suppose that $p>4$.  The strict part of
\cref{thm:q6-scalar-all-high} shows that equality forces $|\delta|=1$.
If $a_1,\ldots,a_n$ are the eigenvalues of $M$ divided by
$\|M\|_{\mathrm F}$, then
$1=|\sum\limits_{i=1}^n a_i^3|\le\sum\limits_{i=1}^n|a_i|^3
\le\sum\limits_{i=1}^n a_i^2=1$.
Thus every nonzero $a_i$ has modulus one.  Since their squares sum to one,
exactly one eigenvalue is nonzero, so $M$ has rank one.  Conversely, if $M$
has rank one, the centered quadratic form is a nonzero constant times
$g^2-1$ or $1-g^2$, and equality holds for every $p\ge4$.

Finally, with $x=g^2$, the density of $x$ on $(0,\infty)$ is
$(2\pi)^{-1/2}x^{-1/2}\exp(-x/2)$.  Therefore
\[
 \|g^2-1\|_p
 =\left[\frac1{\sqrt{2\pi}}\int_0^\infty
   |x-1|^p x^{-1/2}\exp(-x/2)\dd x\right]^{1/p}.
\]
This is \eqref{eq:q6-main-constant} and completes the proof.
\end{proof}

\end{document}